\documentclass[a4paper,reqno]{amsart}
\usepackage{a4wide,amsfonts,amssymb,mathptmx}
\usepackage{hyperref}

\newtheorem{thm}{Theorem}
\numberwithin{thm}{section}
\newtheorem{prop}[thm]{Proposition}
\newtheorem{lem}[thm]{Lemma}
\newtheorem{cor}[thm]{Corollary}

\theoremstyle{definition}

\newtheorem{exmp}[thm]{Example}
 
\theoremstyle{remark}
\newtheorem{rem}[thm]{Remark}
\newcommand{\ip}[2]{\ensuremath{({#1}\,|\,{#2})}} 
\providecommand{\BBb}[1]{{\mathbb{#1}}}

\newcommand{\Bcirc}{\overset{\lower 1.5pt%
              \hbox{$@,@,@,@,@,\scriptscriptstyle\circ$}}B{}}
\newcommand{\Binfty}{\overset{\lower 1.5pt%
              \hbox{$@,@,@,@,@,\scriptscriptstyle\infty$}}B{}}
\newcommand{\bigdot}{\mathbin{\raise.65\jot\hbox{$\scriptscriptstyle\bullet$}}}

\newcommand{\B}{{\BBb B}}
\newcommand{\C}{{\BBb C}}
\newcommand{\Cn}{{\BBb C}^n}

\newcommand{\erd}{\overset{\lower 1pt\hbox{\large.}}{e}
                  \overset{\lower 1pt\hbox{\large.}}{r}}

\newcommand{\Fcirc}{\overset{\lower 1.5pt%
               \hbox{$@,@,@,@,@,\scriptscriptstyle\circ$}}F{}}
\newcommand{\fracc}[2]{{
                \textstyle\frac{#1}{\raise 1pt\hbox{$\scriptstyle #2$}}}}

\newcommand{\fracci}[2]{{\frac{#1}{\raise 1pt\hbox{$\scriptscriptstyle #2$}}}}

\newcommand{\im}{\operatorname{i}}
\renewcommand{\Im}{\operatorname{Im}}

\newcommand{\lap}{\operatorname{\Delta}}

\newcommand{\mlap}{-\!\operatorname{\Delta}}

\renewcommand{\Re}{\operatorname{Re}}

\newcommand{\R}{{\BBb R}}

{
\end{minipage}%
\par\medskip\noindent}%
\newcounter{enmcount}\renewcommand{\theenmcount}{{\rm\arabic{enmcount}}}

\newcounter{rmcount}\renewcommand{\thermcount}{{\rm\roman{rmcount}}}
\newenvironment{rmlist}{%
\begin{list}{{\rm(\thermcount)}}{\setlength{\labelwidth}{\leftmargin}%
\usecounter{rmcount}}}{\end{list}}
\newcounter{Rmcount}\renewcommand{\theRmcount}{{\rm\Roman{Rmcount}}}
\newenvironment{Rmlist}{%
\begin{list}{{\rm(\theRmcount)}}{\setlength{\labelwidth}{\leftmargin}%
\usecounter{Rmcount}}}{\end{list}}

\newcommand{\set}[2]{\{\,#1 \mid #2\,\}}
\newcommand{\Set}[2]{\bigl\{\,#1\bigm| #2\,\bigr\}}

\newcommand{\Span}{\operatorname{span}}

\newcommand{\tr}{\operatorname{tr}}

\begin{document}
\title[Log-Convex Decay in Non-Selfadjoint Dynamics]{Characterisation of Log-Convex Decay in Non-Selfadjoint Dynamics}
\author{Jon Johnsen}
\address{Department of Mathematical Sciences, Aalborg University, Skjernvej 4A, DK-9220 Aalborg {\O}st, Denmark}
\email{jjohnsen@math.aau.dk}
\thanks{Supported by the Danish Research Council, Natural Sciences grant no.~4181-00042.%
\\[9\jot]{\tt To appear in Electronic Research Announcements in Mathematical Sciences, vol.\ 25 (2018)}}

\begin{abstract} 
The short-time and global behaviour are studied for an
autonomous linear evolution equation, which is defined by a generator
inducing a uniformly bounded holomorphic semigroup in a Hilbert space.
A general  necessary and sufficient condition is introduced under which the norm of the solution is shown to be a log-convex and
strictly decreasing function of time, and differentiable also at the initial time with a derivative
controlled by the lower bound of the generator, which moreover is shown to be positively accretive. Injectivity of
holomorphic semigroups is the main technical tool.
\end{abstract}
\maketitle

\section{Introduction}
\thispagestyle{empty}
The subject of this note is the global and short-time behaviour of the solutions to 
an autonomous linear evolution equation having a possibly non-selfadjoint generator $-A$.

It is assumed that $A$ is an accretive operator with domain $D(A)$ in a complex Hilbert space $H$,
with norm $|\cdot|$ and inner product $\ip{\cdot}{\cdot}$, and that $-A$ generates a uniformly bounded, holomorphic
$C_0$-semigroup $e^{-zA}$ for $z$ in an open sector of the form $\Sigma_{\delta}=\set{z\in\C}{-\delta< \arg
  z<\delta}$.  Then the ``height'' function
\begin{equation}
  h(t)=|e^{-tA}u_0|
\end{equation}
is studied for the solution $u(t)=  e^{-tA}u_0$ to the following Cauchy problem, where only initial
data $u_0\ne0$ are considered,
\begin{equation}  \label{ivp-id}
   \partial_t u+Au=0 \quad\text{ for $t>0$},\qquad u(0)=u_0\quad\text{ in $H$}.
\end{equation}
The intention is to investigate the algebraic conditions on $A$, which give a \emph{log-convex} decay of $h(t)$.

In a recent article on final value problems by A.-E. Christensen and the author \cite{ChJo18ax}, cf.\
also \cite{ChJo18}, it
was elucidated and proved (except for one remnant) that if $A$ is an elliptic variational operator
and $A$ is \emph{hyponormal}, cf.\ work of Janas~\cite{Jan94}, then in terms of the numerical range and the lower bound 
\begin{equation}
  \nu(A)=\Set{\ip{Ax}{x}}{x\in D(A),\ |x|=1}, \qquad m(A)=\inf \,\Re\nu(A),   
\end{equation}
there is a ``nice'' decay of the height function:
\begin{quote}
   $h(t)$ is \emph{strictly positive, strictly decreasing} and \emph{strictly convex} on the closed half\-line $t\ge0$,
   and $h(t)$ is differentiable at $t=0$, for $|u_0|=1$ generally with $h'(0)\le -m(A)$, though with 
   \begin{equation}
         \label{h'-intro}
    h'(0)=-\Re\ip{Au_0}{u_0} \qquad\text{ if in addition $u_0\in D(A)$}.
   \end{equation}
 \end{quote}
 
First of all this shows how the \emph{short-time} behaviour at $t=0$ via $h'(0)$ is specifically controlled by  $\nu(A)$, the
numerical range of $A$, and not by its spectrum $\sigma(A)$; whereas the crude decay estimate $h(t)\le
Ce^{-t\underline\sigma}$ for $t\to\infty$ is given by the spectral abscissa
$\underline\sigma=\inf\Re\sigma(A)$ of $A$, say in case $A^*=A\ge0$. 

Secondly, the \emph{global} behaviour of the height $h(t)$ is 
expressed in its strict decrease and strict convexity: even if $A$ has eigenvalues in
$\C\setminus\R$, as may be the case, they do not induce oscillations in the \emph{size} of the
solution $e^{-tA}u_0$ for such $A$---this is ruled out by strict convexity, which
thus can be seen as a stiffness in the decay of $h(t)$.

The present paper generalises the above-mentioned results of \cite{ChJo18,ChJo18ax} in three ways: First the restriction
to variational generators $-A$ is completely removed. 

Secondly, the additional assumption that $A$ is
hyponormal is replaced by the weaker condition that $A$ satisfies the following for vectors
$x\in D(A^2)$ such that $|x|=1$, 
\begin{equation} \label{A-cond}
   2\big(\Re \ip{Ax}{x}\big)^2\le\Re\ip{A^2x}{x}+|Ax|^2.
\end{equation}
The third improvement is the stronger conclusion that $h(t)$ actually is log-convex%
\footnote{A fortunately inconsequential flaw in the argument given for the strict convexity in
\cite{ChJo18ax} is pointed out in Remark~\ref{A2-rem} below. A remedy of this is provided by means of 
the present more general results.%
}.
In fact condition \eqref{A-cond} \emph{characterises} the $A$ for which $h(t)$ is
log-convex; cf.~Theorem~\ref{logcon-thm} below.

Somewhat surprisingly, strict monotone decay $h(t)\searrow0$ for $t\to\infty$ \emph{results}
from (log-)convexity of $h$ (since $e^{-tA}$ is uniformly bounded), hence follows whenever the
generator $A$ fulfils \eqref{A-cond}. The convexity of $h$ then implies existence of $h'(0)=\inf
h'<0$ and that \eqref{h'-intro} holds. The latter shows that $A$ is barely better 
than accretive ($m(A)\ge0$) in the sense that its numerical range is contained in the \emph{open} right half-plane,
\begin{equation} \label{C+-id}
  \nu(A)\subset \set{z\in\C}{\Re z>0} =: \C_+.
\end{equation}
It seems appropriate to call $A$ a \emph{positively accretive} operator, when it has the property \eqref{C+-id}. This is a milder
condition on $A$ than strict accretivity ($m(A)>0$) used by Kato \cite{Kat95}. The
elliptic variational generators in \cite{ChJo18,ChJo18ax} are all strictly accretive, but as
described, there is no need to find a substitute assumption for this, as any $A$ satisfying
criterion \eqref{A-cond} automatically is positively accretive.

To shed more light on the log-convexity criterion \eqref{A-cond}, recall that
every $B\in\B(H)$ satisfies $B=X+\im Y$ for unique selfadjoint operators $X, Y\in\B(H)$,
namely $X=\frac12 (B+B^*)$ and $Y=\frac1{2\im}(B-B^*)$. Hence
\begin{align} \label{AXY-id}
  \Re\ip{Bu}{u}&= \Re\ip{Xu+\im Yu}{u}= \ip{Xu}{u},
  \\
  \Re\ip{B^2u}{u}&= \Re\ip{(X^2-Y^2+\im(XY+YX))u}{u}=|Xu|^2-|Yu|^2,
  \\
   |Bu|^2&= |Xu+\im Yu|^2=|Xu|^2+|Yu|^2+2\Im\ip{Xu}{Yu}.
\end{align}
As the terms $\pm|Yu|^2$ cancel when the last two lines are added, 
\eqref{A-cond} reduces for $B=X+\im Y$ in $\B(H)$ to 
  \begin{equation} \label{A-cond''}
        \ip{Xu}{u}^2\le (|Xu||u|)^2 +\Im\ip{Xu}{Yu}|u|^2\qquad\text{for all $u\in H$}.
  \end{equation}
Here it is noteworthy that $Y$ only appears in one term. In view of the Cauchy--Schwarz inequality,
it is clear that the above is fulfilled when $X$ and $Y$ fit so together that the imaginary part is
positive for all $u\in H$. 
In terms of the commutator $[Y,X]=YX-XY$, 
one may write \eqref{A-cond''} equivalently as
  \begin{equation} \label{A-cond'''}
        \ip{Xu}{u}^2\le (|Xu||u|)^2 +\frac12\Im\ip{[Y,X]u}{u}|u|^2\qquad\text{for all $u\in H$}.
  \end{equation}
However, \eqref{A-cond''} and \eqref{A-cond'''} are always violated for certain $B\in\B(H)$ when
$\dim H\ge2$; cf.\ Remark~\ref{Aneg-rem} below.    

So, in other words, criterion \eqref{A-cond} is not fulfilled for every operator
$A$ in $H$, neither for bounded $A$, nor for  $n\times n$-matrices, $n\ge 2$.
 It is therefore envisaged that \eqref{A-cond} can give rise to
interesting examples when $A$ is a suitable realisation of a partial differential operator.

\section{Discussion and Main Results}  \label{disc-sect}
The reader is assumed familiar with semigroup theory, for which the book of Pazy
\cite{Paz83} could be a reference; the simpler Hilbert space case is exposed e.g.\ by Grubb
\cite[Ch.\ 14]{G09}. It is briefly mentioned that there is a bijective
correspondence between the $C_0$-semigroups $e^{-tA}$ in $\B(H)$ that are uniformly bounded, i.e.\ 
$\|e^{-tA}\|\le M$ for $t\ge0$, and holomorphic in some sector $\Sigma_\delta\subset \C$ for 
$\delta\in\,]0,\frac\pi2[\,$, and the
densely defined, closed operators $A$ in $H$ satisfying a resolvent estimate 
$|\lambda|\|(A+\lambda I)^{-1}\|\le C$ for all $\lambda \in \{0\}\cup\Sigma_{\delta+\pi/2}$.

It is classical that, since $\sigma(A)\subset\set{z\in\C}{\Re z\ge\varepsilon}$ for some
$\varepsilon>0$, there is a bound $\|e^{-tA}\|\le M_\eta e^{-t\eta}$ for $t\ge0$,
$0<\eta<\varepsilon$. This yields the crude decay estimate $h(t)\le M_\eta e^{-t\eta}|u_0|$.

\bigskip

Log-convexity may be a new aspect in the context, so the discussion is begun with this.
First it is recalled that for a strictly positive function $f\colon \R\to \,]0,\infty[\,=:\R_+$, log-convexity means
that $\log f(t)$ is convex, that is, for all $r\le  t$ in $\R$ and $0<\theta<1$,
\begin{equation} \label{f-logcon}
  f((1-\theta)r+\theta t)\le  f(r)^{1-\theta}f(t)^\theta.
\end{equation}
As a slight extension, this also makes sense for non-negative functions $f\colon \R\to[0,\infty[\,$.

A classical exercise shows for the intermediate point $s=(1-\theta)r+\theta t$
that one has $\theta =(s-r)/(t-r)$.
Explicitly log-convexity therefore means for the height function that, for $0\le r< s<t$,
\begin{equation} \label{h-logcon}
  \big|e^{-sA}u_0\big| \le \big|e^{-rA}u_0\big|^{1-\frac{s-r}{t-r}} \big |e^{-tA}u_0\big|^{\frac{s-r}{t-r}}.
\end{equation}
The operator $A$ is just a positive scalar if $\dim H=1$, so \eqref{h-logcon} is then an identity because
of the functional equation of the exponential function $e^{-tA}$ (whereas its slightly weaker
property of strict convexity is expressed via a sharp inequality, oddly enough). 
For  $\dim H>1$ the inequality \eqref{h-logcon} is by no means obvious for the operator function
 $e^{-tA}$ in $\B(H)$;  it is the main subject of this paper.
 
 It is noteworthy that the power function $t\mapsto t^\theta$
 in \eqref{h-logcon} does not require its continuous extension to $t=0$, for since $e^{-tA}u_0$ 
 is holomorphic, the height function fulfils $h(t)>0$, or equivalently $e^{-tA}u_0\ne0$, for $t\ge0$.

 This follows from the restriction to $u_0\ne0$ and the crucial fact that $e^{-zA}$ is an
 \emph{injection} for all $z\in \Sigma_\delta$:
 
 \begin{lem}[\cite{ChJo18ax}] \label{inj-lem}
 Whenever $-A$ generates a holomorphic semigroup $e^{-zA}$ in $\B(X)$ for some complex Banach space
 $X$, and $e^{-zA}$ is holomorphic in the open sector $\Sigma_{\delta}\subset \C$ given by $|\arg z|<\delta$ for some  
 $\delta>0$, then the operator $e^{-zA}$ is injective on $X$ for each such $z$.
 \end{lem}
 
The injectivity is for $t>0$ clearly equivalent to the geometric property that two solutions
$e^{-tA}v$ and $e^{-tA}w$ to the differential equation $u'+Au=0$ cannot have any points of confluence in $X$ when $v\ne w$. 
One obvious consequence of this is the backward uniqueness of $u'+Au=0$; i.e.\ $u(T)=0$ implies $u(0)=0$. 
But injectivity was seemingly first obtained in \cite{ChJo18ax},
cf.\ the elementary proof in Proposition~1 there, using unique analytic continuation. 
\cite[Cor.\ 4.3.9]{Rau91} is analogous, but is given for the Laplacian $A=\mlap$ on
Euclidean space, though for local vanishing of $e^{t\lap}u_0$ in an open set at a fixed time $t>0$.
(An early attempt to obtain Lemma~\ref{inj-lem} was made in a special case in \cite{Sho74}, but
it had flaws pointed out in \cite{ChJo18ax}.)

Injectivity of $e^{-tA}$ is also a crucial tool for the proof of the log-convexity in the present paper.
Indeed, the fact that $h(t)>0$ allows an application of the next result, that characterises the
log-convex $C^2$-functions as the solutions to a differential inequality: 

\begin{lem} \label{logconvex-lem}
If $f$ is continuous $[0,\infty[\,\to \R_+$ and $C^2$ for $t>0$,
the following properties are equivalent:
\begin{Rmlist}
  \item For $0<t<\infty$ it holds true that
  \begin{equation}
                                      f'(t)^2\le f(t)f''(t).
  \end{equation}
  \item $f(t)$ is log-convex on the open half\-line $\,]0,\infty[\,$, that is,
  \begin{equation}
  f(s) \le f(r)^{\frac{t-s}{t-r}}  f(t)^{\frac{s-r}{t-r}} \quad\text{ for $0< r<s<t<\infty$}.
\end{equation}
\end{Rmlist}
In the affirmative case $f(t)$ is log-convex also on the closed half\-line $[0,\infty[\,$.
\end{lem}

\begin{rem}
   It is classical that a $C^2$-function $f$ is convex if and only if 
  $f''\ge0$. This positivity is fulfilled if $f$ satisfies
  (I), as $(f')^2\ge0$ and $f(t)>0$ is assumed; and it is so in qualified way, equivalent to
  log-convexity by Lemma~\ref{logconvex-lem}. 
  Though the lemma is not mentioned, convexity is amply elucidated in \cite{NiPe06}.
\end{rem}

\begin{proof}
By the assumptions $F(t)=\log f(t)$ is defined for $t\ge0$ and $C^2$
for $t>0$, as the Chain Rule gives 
\begin{equation}
  F''(t)=\left (\frac{f'(t)}{f(t)}\right)'=\frac{f''(t)f(t)-f'(t)^2}{f(t)^2}.
\end{equation}
Hence (I) is equivalent to $F''(t)\ge 0$ for $t>0$, which is
the criterion for the $C^2$-function $F$ to be convex for $t>0$;
which is a paraphase of the condition (II) for log-convexity of the positive function
$f(t)$ for $t>0$.  

By letting $r\to 0^+$ for fixed $s<t$, it follows from the continuity of $f(r)$ and 
of $\exp({\frac{t-s}{t-r}}\log f(r))$, that the
inequality in (II) is valid for $0\le r<s<t$. Consequently $f$ is log-convex on the closed
half\-line $[0,\infty[\,$. 
\end{proof}

To shed light on the lemma's consequences for height functions, one may conveniently use 
differential calculus in Banach spaces as exposed e.g.\ by H{\"o}rmander \cite[Ch.\ 1]{H} or Lang \cite{Lan}. Note, though, that the
inner product on $H$, despite its sesquilinearity, is differentiable on the induced real vector space $H_\R$, or 
rather on $H_R\oplus H_\R$,
with derivative $\ip{\cdot}{y}+\ip{x}{\cdot}$ at $(x,y)\in H_\R\oplus H_\R$. 
Since $u(t)=e^{-tA}u_0$ is non-zero for all $u_0\ne0$ by injectivity of the semigroup, cf.\ Lemma~\ref{inj-lem},  there is a  composite 
map between open sets $\R_+\to H_\R\oplus H_\R\to \R_+\to \R_+$  given by $t\mapsto \sqrt{\ip{u(t)}{u(t)}}$.
So the Chain Rule for real Banach spaces gives, since $u'=-Au$ for $t>0$, that
\begin{equation} \label{h'-id}
  h'(t)= \frac{\ip{u'}{u}+\ip{u}{u'}}{2\sqrt{\ip{u}{u}}} = -\frac{\Re\ip{Au}{u}}{|u|};
\end{equation}
and hence, since $u''=(e^{-tA}u_0)''=A^2e^{-tA}u_0=A^2u$,
\begin{equation}  \label{h''-id}
  h''(t)= \frac{\ip{A^2u}{u}+2\ip{Au}{Au}+\ip{u}{A^2u}}{2|u|}-\frac{(\Re\ip{Au}{u})^2}{|u|^3}.
\end{equation}
The differential inequality in (I) of Lemma~\ref{logconvex-lem}, 
\begin{equation} \label{h-ineq}
  (h'(t))^2\le h''(t) h(t),  
\end{equation}
is therefore equivalent to
\begin{equation} 
  \frac{(\Re\ip{Au}{u})^2}{|u|^2}\le   \Re\ip{A^2u}{u}+\ip{Au}{Au} -\frac{(\Re\ip{Au}{u})^2}{|u|^2};
\end{equation}
and to
\begin{equation} \label{h-ineq'}
  2(\Re\ip{Au}{u})^2\le   \Big(\Re\ip{A^2u}{u}+|Au|^2\Big)|u|^2.
\end{equation}
Obviously this condition is fulfilled for every $t>0$ when $A$ satisfies condition \eqref{A-cond}
above, for $u(t)=e^{-tA}u_0$ belongs to the subspace $D(A^n)\subset D(A^2)$ for every $n\ge2$, and all $u_0\in H$, when the semigroup
is holomorphic.
So in this case, it follows from Lemma~\ref{logconvex-lem} that $h(t)=|e^{-tA}u_0|$ is log-convex
for $t\ge0$, for  the continuity of $h(t)$ and of its derivatives given above entail that the
$C^2$-condition is fulfilled.

Conversely, in case the height function $h(t)$ is known to be log-convex for all $u_0\ne0$, then the generator $-A$
necessarily fulfils condition \eqref{A-cond} above. Indeed, in view of the equivalence of
\eqref{h-ineq} and \eqref{h-ineq'}, the former of these holds by the log-convexity of $h$, and so
does the latter. Especially it is seen by insertion of an arbitrary $u_0\in D(A^2)$ in
\eqref{h-ineq'} and commutation of $A$ and $A^2$ with the semigroup that
\begin{equation} 
  2(\Re\ip{e^{-tA}Au_0}{e^{-tA}u_0})^2\le   \Big(\Re\ip{e^{-tA}A^2u_0}{e^{-tA}u_0}+|e^{-tA}Au_0|^2\Big)|e^{-tA}u_0|^2.
\end{equation}
By passing to the limit for $t\to 0^+$ it follows for reasons of continuity that
\begin{equation} 
  2(\Re\ip{Au_0}{u_0})^2\le   \Big(\Re\ip{A^2u_0}{u_0}+|Au_0|^2\Big)|u_0|^2.
\end{equation}
Hence a normalisation to $x=\frac1{|u_0|}u_0$ yields \eqref{A-cond} for every unit vector 
$x$ in $D(A^2)$. Altogether this shows that \eqref{A-cond} characterises the generators $-A$ of
uniformly bounded, holomorphic semigroups having log-convex height functions for all non-trivial initial data.

The log-convexity criterion \eqref{A-cond} should be compared to the sufficient condition $h''(t)>0$ for strict
convexity. The latter is seen at once from the above arguments to be equivalent to the property 
\begin{equation} \label{A-strictly}
  (\Re\ip{Ax}{x})^2<   \Re\ip{A^2x}{x}+|Ax|^2, \qquad\text{ for $x\in D(A^2)$, $|x|=1$},
\end{equation}
where in comparison to \eqref{A-cond} the inequality is strict and a factor of 2 is absent on the left-hand side.

This clearly indicates that log-convexity is \emph{stronger} than strict convexity for non-constant functions:

\begin{lem} \label{logcon-lem}
  When $f\colon I\to [0,\infty[\,$ is log-convex on an interval or half\-line $I\subset\R$, then
  $f$ is convex---and if $f$ is not constant in any subinterval, then $f$ is \emph{strictly} convex
  on $I$.
\end{lem}
\begin{proof}
  Convexity on $I$ follows from \eqref{f-logcon} and Young's inequality for the dual exponents $1/\theta$ and
  $1/(1-\theta)$:
  \begin{equation}
    f((1-\theta)r+\theta t)\le f(r)^{1-\theta}f(t)^\theta \le (1-\theta)f(r)+\theta f(t).
  \end{equation}

In case $f(r)\ne f(t)$, then the last inequality is strict, as equality holds in Young's inequality if
and only if the numerators are identical (cf.\ \cite[p.\ 14]{NiPe06}). This yields the inequality of strict convexity in this
case. 

If there is a common value $C=f(r)=f(t)$ for some $r<t$ in $I$, there is by assumption a
$u\in\,]r,t[\,$ so that $f(u)\ne f(r)$, and because of the convexity of $f$ this entails
that $f(u)<f(r)=f(t)$: when $r<s\le u$ one may write $s=(1-\theta)r+\theta u$ and $s=(1-\omega)r+\omega t$ for suitable
$\theta,\omega\in\,]0,1[\,$, so clearly 
\begin{equation}
  \begin{split}
  f(s)&\le(1-\theta)f(r)+\theta f(u)
\\
      &<(1-\theta)f(r)+\theta f(t) = C= (1-\omega)f(r)+\omega f(t);  
  \end{split}  
\end{equation}
similarly for $u\le s<t$; so $f$ is strictly convex.
\end{proof}

For completeness it is noted that for example $f(t)=e^{t}-1$ is convex, but not log-convex as
$(\log f)''<0$. However, when $f\colon I\to \,]0,\infty[\,$ is log-convex, so is the stretched function
defined for $a<b$ in $I$ as
\begin{equation}
  f_{a,b}(t)=
  \begin{cases}
    f(t)\quad\text{for $t<a$}, \\ f(a)\quad \text{for $a\le t<b$}, \\ f(t-b)\quad\text{for $b\le t$}.
  \end{cases}
\end{equation}
This follows from the geometrically obvious fact that the convexity of $\log f$ survives the
stretching. Since $f_{a,b}$ clearly is not strictly convex, the last assumption of
Lemma~\ref{logcon-lem} is necessary.

When $A$ does satisfy condition \eqref{A-cond}, so that $h(t)$ is log-convex on $[0,\infty[\,$ for
every $u_0\ne0$ (cf.\ the last part of Lemma~\ref{logconvex-lem}), 
then $h(t)$ is necessarily strictly decreasing on $[0,\infty[\,$: the decay estimate $h(t)\le
Ce^{-t\eta} $ and the mere convexity statement in Lemma~\ref{logcon-lem} show that $h$ then 
satisfies the assumptions in the following self-suggesting

\begin{lem} \label{decay-lem}
  If $f\colon [0,\infty[\,\to \R_+$ is convex and $f(t)\to 0$ for
  $t\to\infty$, then $f$ is \emph{strictly} monotone decreasing.
\end{lem}
\begin{proof}
  Given $r<s$ in $[0,\infty[\,$, then $0<f(t_0)<\frac12 f(r)$ holds for some $t_0>s$. Taking
  $\ell(t)=\alpha t+\beta$ so that $\ell(r)=f(r)$ and $\ell(t_0)=f(t_0)$, the fact 
  $\alpha<0$ and convexity on $[r,t_0]$ yield $f(s)\le \ell(s)<\ell(r)=f(r)$. 
\end{proof}

Consequently $h(t)=|e^{-tA}u_0|$ is strictly decreasing on $[0,\infty[\,$ (hence has $h'(t)<0$ for $t>0$).
Therefore $h$ attains a unique global maximum at $t=0$.
Moreover, as $h$ cannot be constant in any subinterval, $h$ is
a strictly convex function on $[0,\infty[\,$, according to Lemma~\ref{logcon-lem} and the log-convexity.

By the convexity of $h$ one has that  $h''(t)\ge0$ for $t>0$, so $h'(t)$ is increasing on $\,]0,\infty[\,$. 
Consequently $\lim_{t\to 0^+ }h'(t)=\inf_{t>0} h'$ exists and belongs to $[-\infty,0[\,$, as $h'<0$. 
By the Mean Value Theorem there is some $t'\in\,]0,t[\,$ so that
\begin{equation}  \label{h't'-id}
  (h(t)-h(0))/t=h'(t')<0.
\end{equation}
This implies that $h(t)$ is (extended) differentiable from the right at $t=0$, with $h'(0)=\inf h'$. 
Since the strong continuity and strict decrease of $h$ gives  $|e^{-tA}u_0|\nearrow 1$ for $t\to 0^+$, 
an application of \eqref{h'-id} yields
\begin{equation}
  h'(0)=\inf h'\le \limsup_{t\to0^+}h'(t)\le \limsup_{t\to0^+} (-m(A)|e^{-tA}u_0|) \le -m(A).
\end{equation}

In case $u_0\in D(A)$ one can exploit that $h'(0)=\lim_{t\to 0^+}h'(t)$ by commuting $A$ with the
semigroup in \eqref{h'-id}, which in the limit gives, because of the strong continuity at $t=0$ and the continuity of inner products, 
\begin{equation} \label{h'0-id}
  h'(0)=\lim_{t\to 0^+} -\Re\ip{e^{-tA}Au_0}{e^{-tA}u_0} 
         = -\Re\ip{Au_0}{u_0} \quad\text{ for $u_0\in D(A)$, $|u_0|=1$}.
\end{equation}
In addition, it is seen from this that $h'(0)$ is a real number for
$u_0\in D(A)$, so $h\in C^1([0,\infty[\,,\R)$ for such $u_0$. 
For general $u_0\in H$ it follows from the Chain Rule that $h\in C^\infty(\R_+,\R)$.

It is also noteworthy that criterion \eqref{A-cond} implies that $A$  is positively
accretive; cf.\ \eqref{C+-id}. Indeed, as $h'(0)<0$ was seen above, \eqref{h'0-id} gives
$\Re\ip{Au_0}{u_0}=-h'(0)>0$ whenever $|u_0|=1$ in $D(A)$; whence $\nu(A)\subset\C_+$.

The above discussion can now be summed up as the main result of this article:

\begin{thm} \label{logcon-thm}
  When $-A$ denotes a generator of a uniformly bounded, holomorphic $C_0$-semigroup $e^{-tA}$ in a
  complex Hilbert space $H$, then the following properties are equivalent:
  \begin{rmlist}
    \item For every $x\in D(A^2)$ with $|x|=1$,
      \begin{equation} \label{A-cond'}
        2\big(\Re\ip{Ax}{x}\big)^2\le \Re\ip{A^2x}{x}+|Ax|^2.
      \end{equation}
   \item The height function $h(t)=|e^{-tA}u_0|$ is log-convex on $[0,\infty[\,$ for every
     $u_0\ne0$; that is, whenever $0\le r<s<t$,
\begin{equation}  \label{h-logcon''}
  \big|e^{-sA}u_0\big| \le \big|e^{-rA}u_0\big|^{\frac{t-s}{t-r}}  \big|e^{-tA}u_0\big|^{\frac{s-r}{t-r}}.
\end{equation}
  \end{rmlist}
In the affirmative case,  the height function $h(t)$ is for $u_0\ne0$ moreover strictly decreasing  
(hence strictly convex) on the closed half\-line
$[0,\infty[\,$ and differentiable from the right at $t=0$, with a derivative in $[-\infty,0[\,$, 
which satisfies
\begin{align}  \label{h'mA-ineq}
  h'(0)&=\inf_{t>0} h'(t)\le - m(A)  \qquad \text{for $|u_0|=1$};
\\
\intertext{and if $u_0\in D(A)$ with $|u_0|=1$, then}
  h'(0)&=-\Re\ip{Au_0}{u_0}
 \label{h'-eq}
\\
  h&\in C^1([0,\infty[\,,\R)\bigcap C^\infty(\R_+,\R)
\end{align}
Furthermore, when $A$ has the properties \upn{(i)} and \upn{(ii)}, then $A$ is positively
accretive, $\nu(A)\subset\C_+$.
\end{thm}

\begin{rem}
  If $A$ is strictly accretive, it is clear that \eqref{h'mA-ineq} is stronger than the
  property  $h'(0)\in[-\infty,0[\,$. Otherwise, when $A$ is merely positively accretive, then
  $h'(0)<0$ may be the stronger statement.
\end{rem}

Returning to the case of hyponormal generators considered in \cite{ChJo18ax}, it is first recalled that
a densely defined unbounded operator $A$ in $H$, following Janas~\cite{Jan94}, is said 
to be hyponormal if  
\begin{equation} \label{hypo-id}
  D(A)\subset D(A^*),\qquad |Ax|\ge |A^*x| \qquad\text{for all $x\in D(A)$}.
\end{equation}
Obviously this is fulfilled if $A^*=A$, but the hyponormal operators extend the
selfadjoint operators in another direction than symmetric ones do (as these have a full operator
inclusion $A\subset A^*$). Since clearly $A$ is normal if and only if both $A$ and $A^* $ are hyponormal,
this operator class is quite general. 

In case $A$ is a hyponormal operator in $H$, the inclusion $D(A)\subset D(A^*)$ gives at once
for $x\in D(A)$ that
\begin{equation}
  2\Re\ip{Ax}{x}=\ip{Ax}{x}+\ip{x}{Ax}=\ip{(A+A^*)x}{x}.
\end{equation}
Invoking also the norm inequality from the definition of hyponormality, a similar reasoning shows for $x\in D(A^2)$,
since $D(A^2)\subset D(A)\subset D(A^*)$, that
\begin{equation}
  |(A+A^*)x|^2= |Ax|^2+\ip{Ax}{A^*x}+\ip{A^*x}{Ax}+|A^*x|^2
  \le 2|Ax|^2+2\Re\ip{A^2x}{x}.
\end{equation}
Hence, by using the Cauchy--Schwarz inequality in the above identity, one finds
\begin{equation} \label{hypo-ineq}
  2(\Re\ip{Ax}{x})^2=\frac12\ip{(A+A^*)x}{x}^2
 \le \frac12|(A+A^*)x|^2|x|^2 
  \le \big(|Ax|^2+\Re\ip{A^2x}{x}\big)|x|^2.
\end{equation}
After a normalisation to $|x|=1$, this shows that a hyponormal operator always fulfils condition (i) in 
Theorem~\ref{logcon-thm}, cf.\ \eqref{A-cond'}.
Therefore one has the following generalisation of \cite{ChJo18ax} to the case of hyponormal non-variational generators:

\begin{cor} \label{hypo-cor}
 Let $-A$ generate a uniformly bounded holomorphic $C_0$-semigroup $e^{-tA}$ in a
  complex Hilbert space $H$. If $A$ is hyponormal, cf.\ \eqref{hypo-id}, then $A$ fulfils
  $\nu(A)\subset\C_+$ and the equivalent
  conditions \upn{(i)} and \upn{(ii)} in Theorem~\ref{logcon-thm}, and consequently the height function
  $|e^{-tA}u_0|$ has all the properties of log-convexity, strict convexity and strict decrease
  together with differentiablity at $t=0$ given in the theorem. 
\end{cor}

However, true hyponormality only exists outside the realm of matrices and
Hilbert--Schmidt operators:

\begin{rem} \label{hyp-rem}
For $A\in \B(H)$ hyponormality means that $\ip{(A^*A-AA^*)x}{x}\ge0$ for all $x$, i.e.\
the commutator is positive, $[A^*,A]\ge0$. For such $A$ the trace
$\tr([A^*,A])$ is defined, and if $A$ is a Hilbert--Schmidt operator, so that
$A^*A$ and $AA^*$ are of trace class, $\tr([A^*,A])=\tr(A^*A)-\tr(AA^*)=0$.
Since $\|T\|\le \tr(T)$ when $T$ is positive, it follows that $[A^*,A]=0$ in
$\B(H)$. (Cf.\ \cite[Section 3.4]{GKP} for these facts.)  Hence every hyponormal Hilbert--Schmidt operator is normal.
Especially every hyponormal $n\times n$-matrix is normal.
\end{rem}

It is instructive to review condition \eqref{A-cond'} in case the accretive operator $A$ is variational:
that is, for some Hilbert space $V\subset H$ algebraically, topologically and densely and some
sesquilinear form $a\colon V\times V\to \C$, which is $V$-bounded and $V$-elliptic in the sense
that (with $\|\cdot\|$ denoting the norm in $V$) for some $C_0>0$
\begin{equation}
  \Re a(u,u)\ge C_0\|u\|^2 \qquad\text{for all $u\in V$},
\end{equation}
it holds for $A$ that $\ip{Au}{v}=a(u,v)$ for all
$u\in D(A)$ and  $v\in V$. This framework and Lax--Milgram's lemma on the properties of $A$ is
exposed in \cite[Ch.\ 12]{G09} and \cite[Ch.\ 3]{Hel13}.
It is classical that $-A$ generates a holomorphic semigroup $e^{-tA}$ in $\B(H)$; an explicit proof
is e.g.\ given in \cite[Lem.\ 4]{ChJo18ax}.

For such operators, $\ip{A^2u}{u}=a(Au,u)$ and $|Au|^2=a(u,Au)$ clearly hold for every $u\in D(A^2)$. So with the usual
convention for the ``real'' part, namely that $a_{\Re}(v,w)=\frac12(a(v,w)+\overline{a(w,v)})$ for $v,w\in V$, one has 
\begin{equation}
  \Re\ip{A^2u}{u}+|Au|^2=\Re a(Au,u)+\Re a(u,Au)= 2\Re \big(a_{\Re}(Au,u)\big).
\end{equation}
Thus the log-convexity criterion \eqref{A-cond'} can be stated for $V$-elliptic variational operators in the form of a
comparison of sesquilinear forms, 
\begin{equation}
  \big(\Re a(u,u)\big)^2\le \Re \big(a_{\Re}(Au,u)\big)\ip{u}{u}\qquad\text{ for $u\in D(A^2)$}.
\end{equation}

\begin{exmp}
  To see that variational operators need not be hyponormal, one may take $H=L_2(\alpha,\beta)$, with norm 
  $\|f\|_0=(\int_{\alpha}^{\beta} |f(x)|^2\,dx)^{1/2}$, for reals $\alpha<\beta$ and let 
  $V=\set{v\in H^1(\alpha,\beta)}{u(\alpha)=0}$ be a subspace of the first Sobolev space with norm
  given by $\|f\|_1^2=\int_{\alpha}^{\beta} (|f(x)|^2+|f'(x)|^2)\,dx$ and the  sequilinear form
  \begin{equation}
    a(u,v)=\int_{\alpha}^{\beta} u'(x)\overline{v'(x)}+u'(x)\overline{v(x)}\,dx.
  \end{equation}
 This is clearly $V$-bounded, and also $V$-elliptic: using partial integration and taking the mean
 of the two expressions for $a(u,v)$, one finds  $\Re a(u,u)=\|u'\|_0^2+\frac12|u(\beta)|^2$, so that
 $\Re a(u,u)\ge C_0\|u\|_1^2$ follows for all $u\in V$ and e.g.\ $C_0=\min(\frac12,(\beta-\alpha)^{-2})$ by ignoring
   the last term and applying the Poincar\'e inequality (it is known that a standard proof of this,
   as in e.g.\ \cite[Thm.\ 4.29]{G09}, applies to  the functions in $V$).

 The induced $A$ acts in the distribution space $\mathcal{D}'(\alpha,\beta)$ of Schwartz \cite{Swz66} as $Au=-u''+u'$,
 which is the advection-diffusion operator, having its domain given by a mixed Dirichlet and Neumann condition,
 \begin{equation}
   D(A)=\Set{u\in V}{u\in H^2(\alpha,\beta),\ u'(\beta)=0}=\Set{u\in H^2(\alpha,\beta)}{u(\alpha)=0,\ u'(\beta)=0}.
 \end{equation}
(The pure Dirichlet realisation of $A=-u''+u'$ has been studied at length; cf.\
Chapter 12 in the treatise of Embree and Trefethen \cite{EmTr05}, where use of pseudospectra is the
main tool.)

 Since $A^*$ is induced by the adjoint form $a^*(u,v)=\overline{a(v,u)}$, it is similarly seen that
 $A^*u=-u''-u'$, but here with the domain characterised by a mixed Dirichlet and Robin condition,
 \begin{equation}
   D(A^*)=\Set{u\in H^2(\alpha,\beta)}{u(\alpha)=0,\ u'(\beta)+u(\beta)=0}.
 \end{equation}
 As both $D(A)\setminus D(A^*)\ne \emptyset$ and $D(A^*)\setminus D(A)\ne\emptyset$, it follows from \eqref{hypo-id}
 that neither $A$ nor $A^*$ is hyponormal. This is part of the motivation for the introduction of
 the general condition \eqref{A-cond'} in this paper.
\end{exmp}

\section{Accretive squares}  \label{A2-sect}

The considerations in \cite{ChJo18,ChJo18ax} also dealt with variational operators $A$ that, instead of being hyponormal, have
accretive squares,
\begin{equation} \label{A2+-id}
  \nu(A^2)\subset\overline{\C}_+=\set{z\in\C}{\Re z\ge0}. 
\end{equation}
The discussion in Section~\ref{disc-sect} also extends to such operators without the assumption that $A$ is variational,
albeit only strict convexity is obtained for $h(t)$.

Indeed, when $m(A^2)\ge0$ holds, then it is seen from \eqref{h''-id} and Cauchy--Schwarz' inequality that
\begin{equation} \label{acc-sq}
  h''(t)=\frac{\Re\ip{A^2u}{u}|u|^{2}+|Au|^2|u|^{2}-(\Re\ip{Au}{u})^2}{|u|^{3}}
        \ge \frac{(|Au||u|)^2-(\Re\ip{Au}{u})^2}{|u|^{3}}
        \ge 0.
\end{equation}
Of course the mere convexity of $h$ for $t>0$ is implied by the above inequality $h''(t)\ge0$,
so
\begin{equation} \label{h-conv}
  h(s)\le \frac{t-s}{t-r}h(r)+\frac{s-r}{t-r}h(t) \qquad\text{ for $0<r<s<t$}.
\end{equation}
As $h$ is continuous on $[0,\infty[\,$, this extends to $0\le r<s<t$, so $h$ is convex on $[0,\infty[\,$.
Hence Lemma~\ref{decay-lem} also here applies to $h$, yielding its strict decrease.
The arguments below Lemma~\ref{decay-lem} then apply \emph{verbatim}, which leads to
differentiability at $t=0$ etc.\ of $h(t)$ (skipping the reference to 
Lemma~\ref{logcon-lem} here).
Moreover, this also yields that $A$ is positively accretive.

However, it remains to prove $h(t)$ strictly convex on $[0,\infty[\,$ when $A^2$ is
accretive (because of the factor 2 on the left-hand side of 
\eqref{A-cond'}, this condition is hardly implied by $m(A^2)\ge0$).
In the case $\nu(A^2)\subset\C_+$, clearly
the first inequality in \eqref{acc-sq} is strict, so that $h''(t)>0$ for $t>0$. Thus $h$
is strictly convex for such $A$. 

However, by inspection of the formula above, $h''(t)=0$ is seen to imply
that both $\Re\ip{A^2u}{u}\ge 0$ and $(|Au||u|)^2-(\Re\ip{Au}{u})^2\ge 0$ must hold with equality in the first numerator.
But then the inequalities 
\begin{equation}
  |\Re\ip{Au}{u}|\le |\ip{Au}{u} | \le |Au||u|
\end{equation}
hold with equality. As Cauchy-Schwarz' inequality is an identity only for proportional vectors,
there is some $\lambda=\mu+\im \omega$, $\mu,\omega\in\R$, such that $Au(t)=\lambda u(t)$. Insertion
of this into the equation $h''(t)=0$ yields 
\begin{equation}
  \Re\lambda^2|u|^4+|\lambda|^2|u|^4-(\Re \lambda|u|^2)^2=0,
\end{equation}
which reduces to
\begin{equation}
  \mu^2=0.
\end{equation}
Hence $\lambda=\im\omega$ is an eigenvalue of $A$, as $u(t)=e^{-tA}u_0\ne0$ in view of the
restriction to $u_0\ne0$ and injectivity of $e^{-tA}$; cf.\ Lemma~\ref{inj-lem}. But it was seen
above that $A$ is positively accretive, so it cannot have any eigenvalues on $\im\R$. Consequently $h''(t)>0$ holds for all $t>0$, 
so $h(t)$ is strictly convex for $t>0$.

To extend the strict convexity to the closed half\-line where $t\ge0$, one may conveniently take
recourse to the slope function $S(r,t)=(h(t)-h(r))/(t-r)$. 
Because of the Mean Value Theorem and the strict increase of $h'$, this satisfies $S(r,s)< S(s,t)$
whenever $0<r<s<t$; which is a classical criterion for strict convexity of $h$ on $\,]0,\infty[\,$.
But this sharp inequality extends to the case $r=0$, for by introducing some $r'$ such that $r=0<r'<s<t$, one
finds from the convexity of $h$ on $[0,\infty[\,$ obtained after \eqref{h-conv} that
\begin{equation}
  S(0,s)\le S(r',s)<S(s,t).
\end{equation}
Indeed, the first of these inequalities is valid since the slope function $S(s,t)$ is monotone
increasing in both arguments separately for every convex function on an interval. 
Hence $h$ is strictly convex on $[0,\infty[\,$.

Altogether this proves a result analogous to Theorem~\ref{logcon-thm}, but not quite as strong as
this, for operators $A$ with accretive squares:

\begin{prop}  \label{A2-prop}
  If $-A$ generates a uniformly bounded, holomorphic $C_0$-semigroup $e^{-tA}$ in a
  complex Hilbert space $H$ and $A$ has an accretive square,
  that is 
  \begin{equation}
    \nu(A^2)\subset\overline{\C}_+,   
  \end{equation}
  then if $u_0\ne0$ the height function 
  $h(t)=|e^{-tA}u_0|$ is strictly convex and strictly decreasing on $[0,\infty[\,$, even with
  $h''>0$ for $t>0$,
  and it is differentiable from the right at $t=0$, with a derivative in $[-\infty,0[\,$ satisfying
\begin{equation}
  h'(0)=\inf_{t>0} h'(t)\le - m(A) \qquad \text{for $|u_0|=1$};
\end{equation}
and if $u_0\in D(A)$ with $|u_0|=1$ it holds that $  h'(0)=-\Re\ip{Au_0}{u_0}$
and $h\in C^1([0,\infty[\,,\R)\bigcap C^\infty(\,]0,\infty[\,,\R)$.
Furthermore, $A$ is then positively accretive, that is, $\nu(A)\subset\C_+$.
\end{prop}

Here $h''(t)>0$ was mentioned explicitly, as not all strictly convex functions 
fulfil this (cf.\ $t^4$), whereas in Theorem~\ref{logcon-thm} this property  was  straightforward
from the differential inequality characterising log-convexity. 

The last fact in Proposition~\ref{A2-prop} that $A$ is positively accretive can \emph{post festum} be much sharpened: 
for an accretive operator $A$ to have an accretive square, cf.\ \eqref{A2+-id}, it is \emph{necessary} that $A$ 
has semiangle $\delta\le\frac{\pi}4$, that is, $|\Im z|\le \Re z$ for every $z\in\nu(A)$. 
This was shown already by Showalter \cite[Lem.\ 3]{Sho74}, who 
gave the main lines in the proof of the following

\begin{lem} \label{A2-lem}
  If $A$ is an operator in $H$ so that $A$, $A^2$ are accretive and $\Re\mu<0$ holds for some $\mu$ in the resolvent set $\rho(A)$, then $|\arg z|\le\pi/4$ for all $z\in\nu(A)$.
\end{lem}
\begin{proof}
  First the claim is proved for every bounded operator
  $B\in \B(H)$; here $B=X+\im Y$ for selfadjoint 
  $X$, $Y\in \B(H)$, as noted prior to \eqref{AXY-id}. Since $B$ is accretive,
  $\ip{Xu}{u}=\Re\ip{Bu}{u}\ge0$ holds for $u\in H$, as does
  \begin{equation}
    \Re\ip{B^2u}{u}\ge0 \iff \ip{(X^2-Y^2)u}{u}\ge0\iff  |Yu|^2 \le |Xu|^2  
  \end{equation}
   By the polar decomposition, $Y=US$ holds for a partial isometry $U$ and 
  $S=|Y|=(Y^*Y)^{1/2}$; the latter is positive and fulfils $|Yu|=|Su|$ and
  $|\ip{Yu}{u}|\le\ip{Su}{u}$ for all $u\in H$ [recall that as $S\ge0$, one has 
  $|\ip{Yu}{u}|^2=|\ip{Su}{U^*u}|^2\le \ip{Su}{u}\ip{SU^*u}{U^*u}=\ip{Su}{u}^2$, for
  $USU^*=|Y^*|=|Y|=S$ as $Y$ is selfadjoint, cf.\ \cite[3.2.19]{GKP}]. 
  Exploiting the fact $|Yu|=|Su|$ in the above,  
  $m(B^2)\ge0$ is seen to imply $X^2-S^2\ge0$, which by the well-known operator monotonicity of the square root on positive operators
  implies that $X\ge S$; cf.\ \cite[E3.2.13]{GKP}. When combined with the second fact on $S$, one finds $|\ip{Yu}{u}|\le
    \ip{Su}{u}\le \ip{Xu}{u}$, so that $z=\ip{Bu}{u}$ belongs to the closed sector
    $\overline{\Sigma}_{\pi/4}\subset\C$ given by $|\arg z|\le \pi/4$.

  For general accretive $A$, the assumption on $\rho(A)$ implies, since $\nu(A)\subset\overline{\C}_+$ that every 
  $\lambda$ having $\Re\lambda<0$ belongs to $\rho(A)$; cf.\ the proof of \cite[Thm.\ 3.9]{Paz83}. 
  Therefore the resolvent $B=(A+\varepsilon I)^{-1}$ is in $\B(H)$ for
  all $\varepsilon>0$, and for $v=B^2u\in D(A^2)$,
  \begin{equation}
    \Re\ip{B^2u}{u}=\Re\ip{(A+\varepsilon I)^2v}{v}=\varepsilon^2|v|^2+2\varepsilon \Re\ip{Av}{v}+\Re\ip{A^2v}{v}.
  \end{equation}
  Now, as $A$, $A^2$ are accretive, so is $B^2$ for $\varepsilon>0$;
  whilst $\ip{Bu}{u}=\varepsilon|v|^2+ \overline{\ip{Av}{v}}$ yields that $B$ is
  accretive. So by the above, $\ip{Bu}{u}\in\overline{\Sigma}_{\pi/4}$ for any $\varepsilon>0$; hence, by
  the formula, $\ip{Av}{v}$ must belong to $\overline{\Sigma}_{\pi/4}$ too.
\end{proof}

As motivation for stating Lemma~\ref{A2-lem} and giving a concise proof
(without the assumption, made in \cite{Sho74}, that $-A$ should
generate a $C_0$-semigroup), it should be mentioned that, contrary to the claim in \cite{Sho74}, having
semiangle $\delta\le\pi/4$ does \emph{not} suffice for $A^2$ to be accretive.

This inaccuracy was pointed out by means of the counter-example in
\cite[Rem.~9]{ChJo18ax}, which is slightly reformulated here for a better reading and in order to note
explicitly that the constructed operator gives rise to a contraction semigroup by the Lumer--Philips theorem, cf.\ 
\cite[Cor.\ 14.11]{G09} or \cite[Thm.\ 4.3]{Paz83}:

\begin{exmp} \label{Show-ex}
To obtain an operator $A$ so that $m(A^2)<0<m(A)$ and $\nu(A)\subset\overline{\Sigma}_{\pi/4}$, 
it suffices to take $A$ in $\B(H)$ if $\dim H\ge2$: 
As $A=X+\im Y$ for selfadjoint $X$, $Y\in\B(H)$,  cf.\ \eqref{AXY-id},
clearly $m(A)=m(X)$. Here $X$ can just be chosen to have 
two orthonormal eigenvectors $v_1$, $v_2$ with eigenvalues $\lambda_2>\lambda_1>0$ and $X=I$ on
$H\ominus\Span(v_1,v_2)$, if this is non-trivial. Then $m(X)=\min(1,\lambda_1)>0$. 
Obviously $\nu(A)\subset\overline{\Sigma}_{\pi/4}$ means that
$|\ip{Yv}{v}|\le\ip{Xv}{v}$ for all $v\in H$, or that  $-X\le Y\le X$.
This is achieved for $Y=\delta X+\lambda_1 U$ if $\delta>0$ is small enough and $U$ is a
partial isometry interchanging $v_1$ and $v_2$; with $U=0$ on $H\ominus\Span(v_1,v_2)$.  
In fact, writing $v=c_1v_1+c_2v_2+v_\perp$ for $v_\perp\in
H\ominus\Span(v_1,v_2)$, since $v_1\perp v_2$, the inequalities for $Y$ are equivalent to 
$2\lambda_1|\Re(c_1\bar c_2)|\le \lambda_1(1-\delta)|c_1|^2+(1-\delta)\lambda_2|c_2|^2
+(1-\delta)|v_\perp|^2$, which by Young's inequality is
assured if $1/(1-\delta)\le (1-\delta)\frac{\lambda_2}{\lambda_1}$, that is if $0<\delta\le
1-\sqrt{\lambda_1/\lambda_2}$. 
Now, $m(A^2)\ge0$ means that $|Xv|^2\ge |Yv|^2$ for all $v$ in $H$, but 
this is \emph{always} violated, as one can see from 
$|Yv|^2= \delta^2|Xv|^2+\lambda_1^2|Uv|^2+2\delta\lambda_1\Re\ip{Xv}{Uv}$ by inserting $v=v_1$, 
whereby the last term vanishes as $v_1\perp v_2=Uv_1$: this leads to
\begin{equation}
  |Xv_1|^2-|Yv_1|^2
  = \lambda_1^2|v_1|^2-(\delta^2\lambda_1^2|v_1|^2+\lambda_1^2|v_2|^2)= -\delta^2\lambda_1^2<0 .
\end{equation}
Specifically the symmetric, but non-normal matrices $    A= 
   \left(\begin{smallmatrix}\lambda&0\\0&4\lambda\end{smallmatrix}\right) 
  + \im\lambda\left(\begin{smallmatrix}\delta& 1\\ 1 & 4\delta\end{smallmatrix}\right)$
are counter-examples in $\B(\C^2)$ for $\lambda>0$, $0<\delta\le 1/2$.
\end{exmp}

\begin{rem} \label{Aneg-rem}
The bounded operator $A$ in Example~\ref{Show-ex} is also useful in relation to
hyponomality and the log-convexity criterion \eqref{A-cond'} in Theorem~\ref{logcon-thm}. 
Here it is shown explicitly that it does not have these properties (for hyponormality this also follows 
from Remark~\ref{hyp-rem}).
The notation from Example~\ref{Show-ex} is continued here.

First $A=X+\im Y=(1+\im\delta)X+\im \lambda_1 U$ entails 
$[A^*,A]=2\im[X,Y]=2\lambda_1\im[X,U]$, so  $w=c_1v_1+c_2v_2$ gives $[A^*,A]w=2\im\lambda_1((\lambda_1-\lambda_2)c_2v_1+(\lambda_2-\lambda_1)c_1v_2)$. 
Hence $[A^*,A]\ne0$, so $A$ is non-normal in $\B(H)$. 

Now, 
$\ip{[A^*,A]w}{w}=4\lambda_1(\lambda_1-\lambda_2)\Im(c_1\bar c_2)$,
and inserting $c_1=\pm \im c_2=2^{-1/2}$ yields  that $\nu([A^*,A])$ contains $\pm 2\lambda_1(\lambda_2-\lambda_1)\ne0$;
hence $[A^*,A]$, $[A,A^*]$ are non-positive. So neither $A$ nor $A^*$ is hyponormal.

Furthermore, for $A$ the criterion for bounded operators in \eqref{A-cond''} is that, for $v\in H$, 
  \begin{equation} 
        \ip{Xv}{v}^2\le (|Xv||v|)^2 +\Im\ip{Xv}{\delta Xv +\lambda_1Uv}|v|^2
        = (|Xv||v|)^2 +\lambda_1\Im\ip{Xv}{Uv}|v|^2.
  \end{equation}
Here it is obvious that $\delta$ is absent in the criterion. To show that the inequality is violated for any choice
of $\lambda_2>\lambda_1>0$ it suffices to insert vectors of the form $v=\im sv_1+v_2$ for $s>0$. 
Indeed, $|v|^2=s^2+1$ due to the orthogonality, and $Xv=\im s\lambda_1v_1+\lambda_2v_2$ 
while $Uv=v_1+\im sv_2$, so the above gives for this $v$,
\begin{equation}
  (s^2\lambda_1+\lambda_2)^2\le  (s^2+1)(s^2\lambda_1^2+\lambda_2^2+s\lambda_1(\lambda_1-\lambda_2)).
\end{equation}
 As the fourth order  term $\lambda_1^2s^4$ cancels on both sides, the term of highest
 degree is $s^3\lambda_1(\lambda_1-\lambda_2)$ on the right-hand
 side. After division by $s^3$ and passage to the limit $s\to\infty$, one therefore  arrives at the
 false statement  ``$0\le \lambda_1(\lambda_1-\lambda_2)$''.
Consequently the operator $A$ from Example~\ref{Show-ex} does not fulfil the log-convexity criterion 
in Theorem~\ref{logcon-thm} for any of the considered values of the parameters. Especially this is
so for the matrix given at the end of Example~\ref{Show-ex}. 
\end{rem}

\section{Final remarks}

\begin{rem} \label{A2-rem}
When $\nu(A^2)\subset\overline{\C}_+$ as in Section~\ref{A2-sect}, it is also illuminating to observe the possibility to depart from
the cleaner expression of the derivatives of the squared height $h(t)^2=|e^{-tA}u_0|^2$:
\begin{equation} \label{h2''-id}
  (h^2)''(t)=(-2\Re\ip{Au}{u})'= 2\Re\ip{A^2u}{u}+2|Au|^2\ge |Ae^{-tA}u_0|^2.
\end{equation}
Here $e^{-tA}$ is injective, and $A$ is so as $\nu(A)\subset\C_+$, whence $u_0\ne0$ yields
$(h^2)''>0$.
Similarly $(h^2)''>0$ can be seen from \eqref{hypo-ineq} to hold if $A$ is hyponormal.
That is, $h^2$ is in both cases strictly convex for $t>0$.

But as $\sqrt\cdot$ is concave (not convex), strict convexity of
$h^2$ is not simply carried over to $h$.
As the task is to prove $h'$ strictly increasing, the formula $h'=(h^2)'/(2\sqrt{h^2})$ looks
convincing as there is strict decrease of the denominator  while the numerator is increasing\,---\,but 
this formula does not lead to the desired conclusion because $(h^2)'<0$. 

This small point was overlooked in \cite[Prop.~4]{ChJo18ax},
yet the statement given there is nevertheless correct. 
Indeed, \cite[Prop.~4]{ChJo18ax} is generalised
to non-variational $A$ having accretive squares in Proposition~\ref{A2-prop}, 
and to non-variational hyponormal generators $A$ in Corollary~\ref{hypo-cor}. A further generalisation to
generators $A$ satisfying the log-convexity criterion \eqref{A-cond'} is provided by Theorem~\ref{logcon-thm}.
\end{rem}

\begin{rem}
  For matrices $A$ in $\B(\Cn)$ the dynamical properties of \eqref{ivp-id} have been studied for decades, and e.g.\
  Perko \cite[Ch.\ 1]{Per01} gave a concise treatment with many explicit formulas for
  the exponential matrix $e^{-tA}$ and the resulting solution $u(t)$. However, most systems have
  eigenvalues that are complicated or even impossible to write down ($n\ge 5$), and this led Moler and Van Loan
  to review the possibilities in 1978 in ``Nineteen dubious ways to calculate the exponential of a matrix'',
  with an update in 2003 \cite{MoVL03}.

  The present results are closer in spirit to more recent work, a glimpse of which is given
  here, following the inspiring exposition of  Embree and Trefethen \cite[Ch.\ 14]{EmTr05}. 
  A major subject of interest has been the behaviour
  of the operator norm $E(t)=\|e^{-tA}\|$, which has the advantage of being independent of any
  initial data $u_0$, thereby letting the influence of especially non-normal matrices shine
  though. At $t=0$ it is a main result that the numerical range $\nu(A)$ controls the growth rate of $E$,
  \begin{equation} \label{E'-id}
    E'(0)=-m(A).
  \end{equation}
  For $t\to\infty$ it
  is known that $\frac1t\log E(t)\to -\underline\sigma$, where again $\underline\sigma=\inf\Re\sigma(A)$ denotes the
  spectral abscissa of $A$; so the long-term behaviour is controlled by 
  $\sigma(A)$. For the transition phase there is the pseudospectral estimate $\sup_{t\ge0}E(t)\ge
  \alpha_\varepsilon(-A)/\varepsilon$,
  supplied with estimates from below of $\sup_{0\le s\le t}E(s)$ that permit an exploration of the
  time $t_0$ at which $\sup_{t\ge0}E(t)$ is attained.

  However, when $u_0$ is reintroduced, the inequality $|e^{-tA}u_0|\le E(t)|u_0|$
  is a crude estimate (a worst-case scenario), which does not suffice to settle whether 
  $h(t)=|e^{-tA}u_0|$ has properties like strict decrease or
  log-convexity as in Theorem~\ref{logcon-thm}. Moreover, 
  \eqref{E'-id} is often highly misleading for the short-time behaviour of $h(t)$ itself. For
  example, for $A=\left(\begin{smallmatrix}-1&0\\0&3\end{smallmatrix}\right) $ one has
  $\nu(A)=[-1,3]$ with $m(A)=-1$, so while $E'(0)=1$ holds by \eqref{E'-id}, indicating a
  \emph{growth} at $t=0$, the choice $u_0=e_2=(0,1)$ gives, because of \eqref{h'-eq},
  \begin{equation}
    h'(0)=-\Re\ip{Ae_2}{e_2}=-3<1=E'(0).
  \end{equation}
 This sharp contrast between the properties of the solutions to $u'+Au=0$ with $u(0)=u_0$ and those of
 $\|e^{-tA}\|$ also motivates the study of the height function $h(t)=|e^{-tA}u_0|$ in the present paper.
\end{rem}

\begin{rem} \label{contraction-rem}
  If the generator $-A$ of the uniformly bounded holomorphic semigroup is dissipative, i.e.\ $A$
  is accretive, then $e^{-tA}$ is a classical contraction semigroup; cf.\ \cite[Cor.\ 14.11]{G09}. 
  That is, $\|e^{-tA}\|\le 1$ holds for the operator norm for $t\ge0$, whence $h(t)\le
  \|e^{-tA}\|\cdot |u_0|\le |u_0|$---i.e.\ an estimate by a constant. If $m(A)>0$ it is also classical that
  $-(A-\varepsilon I)$ for $0<\varepsilon<m(A)$ gives the contractions
  $e^{-t(A-\varepsilon I)}=e^{t\varepsilon}e^{-tA}$, so the sharper estimate $|e^{-tA}u_0|\le
  e^{-t\varepsilon}|u_0|$ holds for $t\ge0$ and any $\varepsilon\in\,]0,m(A)[\,$, hence also for
  $\varepsilon=m(A)$.
  But this exponential decay is just a crude estimate that requires
  $m(A)>0$. For comparison it is observed that if $A$ just satisfies the log-convexity
  criterion \eqref{A-cond'}, so that $m(A)=0$ is
  possible and Theorem~\ref{logcon-thm} applies, the log-convex and strictly decreasing
  behaviour of the height function $|e^{-tA}u_0|$ constitutes a rather more precise dynamical
  property of the evolution problem $u'+Au=0$, $u(0)=u_0$.
\end{rem}


\providecommand{\bysame}{\leavevmode\hbox to3em{\hrulefill}\thinspace}
\providecommand{\MR}{\relax\ifhmode\unskip\space\fi MR }
\providecommand{\MRhref}[2]{%
  \href{http://www.ams.org/mathscinet-getitem?mr=#1}{#2}
}
\providecommand{\href}[2]{#2}

\end{document}